\documentclass{amsart}

\usepackage{amsmath}

\usepackage{amssymb}

\usepackage{amsthm}

\usepackage{enumitem}

\usepackage{centernot}

\usepackage{scalefnt}

\usepackage[numbers]{natbib}

\usepackage{graphicx}

\usepackage[all]{xy}

\theoremstyle{definition}
\newtheorem{thm}{Theorem}[section]

\theoremstyle{definition}
\newtheorem{defn}[thm]{Definition}

\theoremstyle{definition}
\newtheorem{ex}[thm]{Example}

\theoremstyle{definition}
\newtheorem{lem}[thm]{Lemma}

\theoremstyle{definition}
\newtheorem{cor}[thm]{Corollary}

\newtheorem{prop}[thm]{Proposition}

\title{Unique Factorization in Polynomial Rings with Zero Divisors}
\author{D.D. Anderson and Ranthony A.C. Edmonds}

\begin{document}

\begin{abstract} Given a certain factorization property of a ring $R$, we can ask if this property extends to the polynomial ring over $R$ or vice versa. For example, it is well known that $R$ is a unique factorization domain if and only if $R[X]$ is a unique factorization domain. If $R$ is not a domain, this is no longer true. In this paper we survey unique factorization in commutative rings with zero divisors, and characterize when a polynomial ring over an arbitrary commutative ring has unique factorization. 
\end{abstract}



\maketitle

\scalefont{1}
\section{Introduction}

Let $D$ be an integral domain. It is well known that the polynomial ring $D[X]$ is a unique factorization domain (UFD) if and only if $D$ is a UFD. Of course $D$ is a UFD if (1) every nonzero nonunit of $D$ is a finite product of irreducible elements and (2) if $a_1 \cdots a_n=b_1 \cdots b_m$ where $a_i,b_j$ are irreducible, then $n=m$ and after re-ordering if necessary $a_i$ and $b_i$ are associates. Equivalently, $D$ is a UFD if each nonzero nonunit of $D$ is a finite product of principal primes.

Suppose that we allow our commutative ring $R$ to have zero divisors. We consider the question: is $R[X]$ a unique factorization ring if and only if $R$ is? Now there are several ways to define a ``unique factorization ring", all of which agree in the domain case. First, there are a number of ways to define ``irreducible" element and the notion of ``associates", see Section 2 for details. We define $R$ to be a (Bouvier\cite{B2}-Galovich\cite{G}) unique factorization ring if (1) and (2) defined in the first paragraph hold, see Definition 5.3. A related type of unique factorization ring, called weak unique factorization rings, were considered in \cite{UFRZD1}. Fletcher \cite{UFR} defined another type of unique factorization ring and several types of ``reduced unique factorization rings" were investigated in \cite{CAV}. We can also consider rings, called factorial rings, in which the nonunit regular elements have unique factorization into irreducibles.

Let $R$ be a commutative ring and $X$ an indeterminate over $R$. The main purpose of this article is to determine when the polynomial ring $R[X]$ has some form of unique factorization. We determine when $R[X]$ is a factorial ring, a unique factorization ring, a weak unique factorization ring, a Fletcher unique factorization ring, or a [strong] ($\mu-)$ reduced unique factorization ring, see Section 5. Unlike the domain case, if a commutative ring $R$ has one of these types of unique factorization, $R[X]$ need not. In Section 6 we examine the good and bad behavior of factorization in $R[X]$ where $R$ is one of these types of unique factorization rings.

In Section 2 we begin with a brief review of factorization in integral domains and commutative rings with zero divisors. The various types of irreducible elements and associate relations are defined. Section 3 reviews some basic facts about polynomial rings that will be used throughout the paper. We also discuss ``irreducible" elements of $R[X]$.

Section 4 involves the factorization of powers of an indeterminate $X$ over a commutative ring $R$. It is shown (Theorem 4.3) that $X$ is a product of irreducible elements (resp., principal primes) if and only if $R$ is a finite direct product of indecomposable rings (resp., integral domains). In the case where $X$ is a product of irreducibles, this factorization is unique up to order and associates (Theorem 4.3) while each power of $X^n$ has unique factorization into irreducibles if and only if $R$ is reduced and a finite direct product of indecomposable rings (Theorem 4.5). 

Throughout this paper all rings will be commutative with an identity. Suppose that $R$ is a commutative ring. We denote the Jacobson radical, nilradical, the set of zero divisors, and the set of idempotents of $R$ by $J(R)$, $nil(R)$, $Z(R)$, and $Id(R)$, respectively. An element is \emph{regular} if it is not a zero divisor.

\

\section{A Brief Review of Factorization}


In this section we first give a very brief review of factorization in an integral domain and then give a slightly longer review of factorization in a commutative ring with zero divisors. 

Let $D$ be an integral domain. Two elements $a,b \in D$ are \emph{associates}, denoted $a \sim b$, if $a \mid b$ and $b \mid a$ which is equivalent to $Da=Db$ or to $a=ub$ for some unit $u \in D$. An element $a \in D$ is \emph{irreducible} or an \emph{atom} if $a$ is a nonzero nonunit and for $b,c \in D$, $a=bc$ implies $b$ or $c$ is a unit of $D$. It is easy to see that for a nonzero nounit $a \in D$, the following conditions are equivalent: (1) $a$ is irreducible, (2) if $a=bc$ for $b,c \in D$, then $a \sim b$ or $a \sim c$, and (3) $Da$ is a maximal element of the set of proper principal ideals of $D$. 

An integral domain is \emph{atomic} if every nonzero nonunit of $D$ is a finite product of atoms while $D$ satisfies the \emph{ascending chain condition on principal ideals (ACCP)} if every ascending chain of principal ideals becomes stationary. It is well known that if $D$ satisfies ACCP then $D$ is atomic, but the converse need not hold. For a review of factorization in an integral domain the reader is referred to \cite{AAZ} and \cite{GHK}.


The terminology and general theory of factorization for commutative rings with zero divisors is less standard. A general approach to factorization in commutative rings is given in \cite{AV1}. Also see \cite{AAV2} and \cite{AV2}. We review some of the details. Let $R$ be a commutative ring. Two elements, $a,b \in R$ are \emph{associates}, denoted $a \sim b$, (resp., \emph{strong associates}, denoted $a \approx b$, \emph{very strong associates}, denoted $a \cong b$) if $a \mid b$ and $b \mid a$, or equivalently $Ra=Rb$ (resp., $a=ub$ for some unit $u \in R$, $a \sim b$ and either $a=b=0$ or $a \neq 0$ and $a=rb$ for $r \in R$ implies that $r$ is a unit in $R$). For $a,b \in R$ we have $a \cong b \implies a \approx b \implies a \sim b$, but none of these implications can be reversed. While $\sim$ and $\approx$ are congruences on the monoid $(R,\cdot)$, $\cong$ is reflexive on $R$ and hence a congruence on $(R,\cdot)$ if and only if $R$ is \emph{pr\`esimplifiable}, that is, each element $x$ of $R$ is \emph{pr\`esimplifiable}: $x=xy$ for $y \in R \implies x=0$ or $y$ is a unit of $R$. 

In \cite{AC2} the last two associate relations were generalized as follows. The elements $a,b \in R$ are \emph{strongly regular associates}, denoted $a \approx_r b$ (resp., \emph{very strongly regular} \emph{associates}, denoted $a \cong_r b$,) if $a=rb$ and $b=sa$ where $r,s \in R$ are regular (resp., $a \sim b$ and either $a=b=0$ or $a \neq 0$ and $a=rb$ for $r \in R$ implies that $r$ is regular.) Finally, $R$ is \emph{weakly pr\`esimplifiable} if for $x, y \in R$, $x=xy$ implies $x=0$ or $y$ is regular.


Using the three different associate relations, we can define three different types of irreducible elements. A nonunit $a \in R$ (with possibly $a=0$) is \emph{irreducible} or an \emph{atom} (resp., \emph{strongly irreducible}, \emph{very strongly irreducible}) if for $a=bc$ with $b,c \in R$, $a \sim b$ or $a \sim c$ (resp., $a \approx b$ or $a \approx c$, $a \cong b$ or $a \cong c$). The nonunit $a \in R$ is \emph{$m$-irreducible} if $Ra$ is a maximal element of the set of proper principal ideals of $R$. Note that the following are equivalent: (1) $R$ is an integral domain, (2) $0$ is prime, (3) $0$ is irreducible, (4) $0$ is strongly irreducible, and (5) $0$ is very strongly irreducible. But $0$ is $m$-irreducible if and only if $R$ is a field. 

Let $R$ be a commutative ring. A nonzero nonunit $p \in R$ is \emph{weakly prime} \cite{UFRZD1} if $p \mid ab \neq 0$, $a,b \in R$, implies $p \mid a$ or $p \mid b$. Certainly a prime element is weakly prime and a weakly prime element is irreducible. Moreover, a weakly prime element $p$ is either prime or satisfies $p^2=0$. For suppose that $p$ is weakly prime and $p^2 \neq 0$. Suppose that $p \mid ab$. If $ab \neq 0$, $p \mid a$ or $p \mid b$. So suppose that $ab=0$. Now $p \mid a(b+p)$, so $p \mid a$ or $p \mid b+p$ and hence $p \mid b$ unless $a(b+p)=0$. But $a(b+p)=0$ gives $ap=0$. Likewise we can assume that $bp=0$. But then $(a+p)(b+p)=p^2\neq 0$. So $p \mid a+p$ or $p \mid b+p$ and hence $p \mid a$ or $p \mid b$. Thus $p$ is prime. Hence a regular weakly prime element is prime. Also, if $R$ is not indecomposable, a nonzero weakly prime element $p$ is prime. For if $p=(p_1,p_2) \in R=R_1 \times R_2$ is weakly prime but not prime, $p^2=0$ gives $p_1^2=p_2^2=0$. Hence $p \mid (p_1,1)(1,p_2)$ but $p \centernot\mid (p_1,1), (1, p_2)$. For more on weakly prime elements, see \cite{WPD}.

For a nonzero element of $R$ we have the following implications, none of which can be reversed:

\

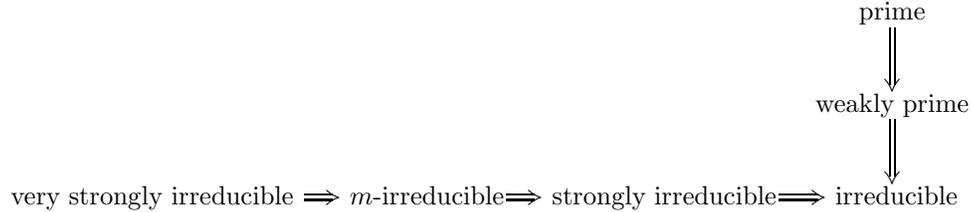
\vspace{-5mm}
\[
\xymatrix@C=0.45cm@1{ 
& & & \text{prime} \ar@{=>}[d]\\
& & &\text{weakly prime} \ar@{=>}[d]\\
\text{very strongly irreducible } \ar@{=>}[r] &\text{ $m$-irreducible} \ar@{=>}[r] &\text{ strongly irreducible} \ar@{=>}[r] &\text{ irreducible}}\\
\] 

\vspace{7mm}

The following theorem summarizes some useful facts about irreducible elements.
\begin{thm} Let $R$ be a commutative ring. 

\noindent
\begin{enumerate}
\item For regular elements, or more generally nonzero pr\`esimplifiable elements, the four types of irreducible elements coincide.
\item For $a \in R$, the following are equivalent:
	\begin{enumerate}
	\item $a$ is irreducible,
	\item there is a prime ideal $P$ of $R$ with $Ra$ a maximal element of the set of principal ideals of $R$ contained in $P$, and 
	\item either (i)  $a$ is regular and $a$ is $m$-irreducible or (ii) $a$ is a zero divisor and $Ra$ is a maximal element of the set of principal ideals of $R$ contained in $Z(R)$.
	\end{enumerate}
\item If $a_1, a_2 \in R$ with $a_1$ irreducible and $Ra_1 \subsetneq Ra_2 \subsetneq R$, then $a_1$ is a zero divisor and $a_2$ is regular.
\item A nonzero nonunit $a \in R$ is very strongly irreducible if and only if for $b,c \in R$ with $a=bc$, either $b$ or $c$ is a unit.
\item Suppose that $0 \neq a \in R$ is very strongly irreducible. Then $ann(a) \subseteq J(R)$. Hence if $J(R)=0$, $a$ is regular.
\item For $0 \neq a \in R$, $a$ is $m$-irreducible if and only if either (i) $a$ is very strongly irreducible or (ii) $Ra$ is an (idempotent) maximal ideal. Thus for $R$ indecomposable, $a$ is $m$-irreducible if and only if it is very strongly irreducible.
\item Let $\{R_{\alpha}\}_{\alpha \in \Lambda}$ be a nonempty family of commutative rings and let $(a_\alpha) \in \prod_{\alpha \in \Lambda} R_{\alpha}$. Then $(a_{\alpha})$ is irreducible (resp., strongly irreducible, $m$-irreduc-\\ ible, prime) if and only if each $(a_{\alpha})$ except for one $\alpha_0 \in \Lambda$ is a unit in $R_{\alpha}$ and that $a_{\alpha_0}$ is irreducible (resp., strongly irreducible, $m$-irreducible, prime.) However, $(a_{\alpha_0})$ is very strongly irreducible if and only if each $a_{\alpha}$ except for one $\alpha_0 \in \Lambda$ is a unit in $R_{\alpha}$ and that $a_{\alpha_0}$ is very strongly irreducible in $R_{\alpha_0}$ and is nonzero unless $|\Lambda|=1$ and $R_{\alpha_0}$ is an integral domain.\\
\end{enumerate}

\end{thm}

\begin{proof}
\
\begin{enumerate}
\item If $x$ is a nonzero pr\`esimplifiable element, then $x \sim y \iff x \cong y$. So $x$ is irreducible if and only if it is very strongly irreducible.
\item $(a) \iff (b)$ \citep[Theorem 2.14]{AV1}. $(a) \iff (c)$ \citep[Corollary 1]{AC1}.
\item \citep[Theorem 1]{AC1}
\item \citep[Theorem 2.5]{AV1}
\item Suppose that $0 \neq a \in R$ is very strongly irreducible. Let $c \in ann(a)$. For $d \in R,$ $a=a(1+dc)$. Hence $1+dc$ is a unit for each $d \in R$; so $c \in J(R)$. 
\item \cite[Theorem 2.9]{CA}.
\item \cite[Theorem 2.15]{AV1}.
\end{enumerate}
\end{proof}


Each of the forms of irreducibility leads to a form of atomicity. The commutative ring $R$ is \emph{atomic} (resp., \emph{strongly atomic}, \emph{very strongly atomic}, \emph{$m$-atomic}, \emph{$p$-atomic}) if each nonzero nonunit of $R$ is a finite product of irreducible (resp., strongly irreducible, very strongly irreducible, $m$-irreducible, prime) elements of $R$. Note that if $R$ is not a domain, then $0$ too is a finite product of the appropriate type of irreducible elements. We collect some facts about atomic rings.\\

\begin{thm} Let $R$ be a commutative ring,
\begin{enumerate}
\item $R$ very strongly atomic $\implies$ $R$ $m$-atomic $\implies$ $R$ strongly atomic $\implies$ $R$ atomic; $R$ $p$-atomic $\implies$ $R$ strongly atomic and $R$ satisfies ACCP; and $R$ satisfies ACCP $\implies$ $R$ is atomic. However, none of these implications can be reversed.
\item Suppose that $R$ is indecomposable. Then $R$ is very strongly atomic $\iff$ $R$ is $m$-atomic.
\item $R$ is $p$-atomic $\iff$ $R$ is a finite direct product of UFDs and SPIRs.
\item Suppose that $0$ is a product of $n$, $n \geq 1$, irreducible elements. Then $R$ is a direct product of at most $n$ indecomposable rings. 
\item Suppose that $\{R_{\alpha}\}_{\alpha \in \Lambda}$ is a nonempty family of commutative rings. If $\prod_{\alpha \in \Lambda} R_{\alpha}$ satisfies ACCP or any of the forms of atomicity, then $\Lambda$ is finite.
\item Let $R_1, \hdots, R_n$ be commutative rings. 
	\begin{enumerate}
	\item $R_1 \times \cdots \times R_n$ satisfies ACCP (resp., is atomic, strongly atomic, $p$-atomic) if and only if each $R_i$ satisfies ACCP (resp., is atomic, strongly atomic, $p$-atomic),
	\item $R_1 \times \cdots \times R_n$ is $m$-atomic if and only if each $R_i$ is $m$-atomic and if $n>1$ and some $R_{i_0}$ is an integral domain, then $R_{i_0}$ must be a field.
	\item $R_1 \times \cdots \times R_n$ is very strongly atomic if and only if each $R_i$ is very strongly atomic and if some $R_{i_0}$ is an integral domain we must have $n=1$.
	\end{enumerate}
\item If $R$ satisfies ACCP (resp., is atomic, strongly atomic, $m$-atomic, very strongly atomic, $p$-atomic), then $R$ is a finite direct product of indecomposable rings satisfying ACCP (resp., which are atomic, strongly atomic, $m$-atomic, very strongly atomic, $p$-atomic).
\end{enumerate}
\end{thm}

\begin{proof}
\
\begin{enumerate}
\item \citep[Theorem 3.7]{AV1}.
\item Theorem 2.1 (6).
\item If $R$ is $p$-atomic, every proper principal ideal of $R$ is a finite product of principal prime ideals. So $R$ is a $\pi$-ring, i.e., a commutative ring in which every proper principal ideal is a product of prime ideals. Hence $R$ is a finite direct product of special principal ideal rings (SPIRs) and $\pi$-domains. Thus $R$ is a finite direct product of SPIRs and $p$-atomic integral domains (=UFDs). The converse is clear. See \citep[Theorem 3.6]{AV1} and the paragraph preceding it.
\item \cite[Theorem 3.3] {AV1} and its proof.
\item \cite[Theorem 3.4]{AV1}.
\item \cite[Theorem 3.4]{AV1}.
\item This follows from (1), (4), and (6).
\end{enumerate}
\end{proof}

\

\section{Some Simple Results about Polynomial Rings}

In this section we collect some simple useful results concerning polynomial rings. The following characterizations of units, zero divisors, idempotents, and nilpotents and of the Jacobson radical and nilradical of $R[X]$ are well known.

\begin{thm}
Let $R$ be a commutative ring and $X$ an indeterminate over $R$. Let $f=a_0+a_1X+\cdots+a_nX^n \in R[X]$. 
\begin{enumerate}
\item $f$ is a unit $\iff$ $a_0$ is a unit and $a_i$ is nilpotent for $i \geq 1$
\item $f$ is a zero divisor $\iff$ there exists $0 \neq c \in R$ with $cf=0$
\item $f$ is idempotent $\iff$ $a_0$ is idempotent and $a_i=0$ for $i \geq 1$
\item $f$ is nilpotent $\iff$ each $a_i$ is nilpotent
\item $J(R[X])=nil(R[X])=nil(R)[X]$
\end{enumerate}
\end{thm}

Since $Id(R)=Id(R[X])$, $R$ is indecomposable if and only if $R[X]$ is indecomposable. Any finite direct product decomposition of $R[X]$ arises from a direct product decomposition of $R$. For a finite direct product decomposition $R=R_1 \times \cdots \times R_n$, we may naturally identify $R[X]$ with $R_1[X] \times \cdots \times R_n[X]$ via the map $(a_i^{(0)})+(a_i^{(1)})X+\cdots+(a_i^{(n)})X^n \rightarrow (a_i^{(0)}+a_i^{(1)}X+\cdots+a_i^{(n)}X^n).$

For an ideal $I$ of $R$ we may also identify $R[X]/I[X]$ with $(R/I)[X]$ via the map $a_0+a_1X+\cdots+a_nX^n+I[X] \mapsto (a_0+I)+(a_1+I)X+\cdots+(a_n+I)X^n$. Now for $a \in R$, we may consider $a$ as an element of $R$ or of $R[X]$, and for example for $a,b \in R$ we may consider $a \cong b$ as elements of $R$ or of $R[X]$. With the obvious notation we write $a \cong_R b$ or $a \cong_{R[X]}b$. The next result collects some simple results.\\

\begin{thm}
Let $R$ be a commutative ring and $X$ an indeterminate over $R$. Let $a,b \in R$ and $f,g \in R[X]$.
\begin{enumerate}
\item $a \sim_R b \iff a \sim_{R[X]} b$
\item $a \approx_R b \iff a \approx_{R[X]} b$
\item $a \cong_{R[X]} b \iff a \cong_R b$ and $a=0$ or $ann(b) \subseteq nil(R)$
\item $f \sim_{R[X]}g \iff f \approx_ {r R[X]} g$
\item $a\sim_R b \iff a \approx_{r R[X]} b$
\item $R[X]$ is pr\`esimplifiable $\iff$ $R$ is pr\`esimplifiable and $0$ is primary.
\item $R[X]$ is weakly pr\`esimplifiable $\iff$ $R$ is weakly pr\`esimplifiable. Hence if $R$ is pr\`esimplifable, $R[X]$ is weakly pr\`esimplifable.
\item $a$ is irreducible as an element of $R$ if and only if $a$ is irreducible as an element of $R[X]$.
\end{enumerate}
\end{thm}

\begin{proof}
\
\begin{enumerate}
\item Clear.
\item Clear. 
\item We may assume $a \neq 0$.
$(\Longleftarrow)$ Suppose $a=fb$ for $f=a_0+a_1X+\cdots+a_nX^n \in R[X]$. Then $a=a_0b$, so $a\cong_R b$ gives $a_0 \in U(R)$. For $i \geq 1$, $a_ib=0$, so $a_i \in ann(b) \subseteq nil(R)$. Hence $f \in U(R[X])$. $(\Longrightarrow)$ Certainly $a \cong_{R[X]} b \implies a \cong_R b$. So $a=rb$ for some $r \in R$. Suppose $c \in ann(b)$. Then $a=(r+cX)b$. So $a \cong_{R[X]} b$ gives $r+cX \in U(R[X])$ and hence $c \in nil(R)$. 
\item This follows from the proof of \citep[Theorem 18 (3)]{AC2}.
\item Combine (1) and (4). (6) \cite{B1} (7) \citep[Theorem 18 (2)]{AC2}.
 \addtocounter{enumi}{2}
\item  \citep[Theorem 6.2]{AV1}.
\end{enumerate}
\end{proof}

If $R$ is an integral domain and $a \in R$ and $f \in R[X]$ with $a \sim f$, then $f \in R$. Suppose that $R$ has a nonzero nilpotent $a$. Then $1+aX \in U(R[X])$, so $1 \sim 1+aX$, in fact, $1 \cong 1+aX$. Thus if $R[X]$ satisfies $a \sim f \implies f \in R$ for $a \in R$ and $f \in R[X]$, $R$ must be reduced. The converse is also true.

\begin{prop} 
For a commutative ring $R$ and indeterminate $X$ over $R$, the following are equivalent. 
\begin{enumerate}
\item $R$ is reduced.
\item For $a \in R$ and $f \in R[X]$, $a \sim f \implies f \in R$.
\item For $a \in R$ and $f \in R[X]$, $a \approx f \implies f \in R$.
\item For $a \in R$ and $f \in R[X]$, $a \cong f \implies f \in R$.
\end{enumerate}
\end{prop} 

\begin{proof}
$(1) \Longrightarrow (2)$ Suppose that $R$ is reduced and $a \sim f=a_0+a_1X+\cdots+a_nX^n \in R[X]$. Let $P$ be a prime ideal of $R$. Then in $\bar{R}=R/P$, $\bar{a}\sim \bar{f}=\bar{a_0}+\bar{a_1}X+\cdots+\bar{a_n}X^n$. Since $\bar{R}$ is an integral domain, $\bar{a_1}=\cdots=\bar{a_n}=\bar{0}$. So for $i \geq 1$, $a_i \in \bigcap \{P \mid P \in Spec(R) \}=nil(R)=0$. So $f \in R$. 
$(2) \implies (3) \implies (4)$ Clear. $(4) \implies (1)$ This was shown in the preceding paragraph. 

\end{proof}

Let $a \in R$. We have seen that $a$ is irreducible as an element of $R$ if and only if $a$ is irreducible as an element of $R[X]$ and certainly the same holds for ``prime". However, the next example (essentially \citep[Example 6.1]{AV1}) shows that this does not hold for the other forms of irreducibility. Indeed, $a$ can be very strongly irreducible (and prime) as an element of $R$ but not even strongly irreducible as an element of $R[X]$. But it is easily checked that if $a \in R$ is strongly irreducible, very strongly irreducible, $m$-irreducible, or weakly prime as an element of $R[X]$, then it has the corresponding property as an element of $R$. We do not know whether a weakly prime element of $R$ is weakly prime as an element of $R[X]$. 

\begin{ex}
Let $R=\mathbb{Z}_{(2)}(+)\mathbb{Z}_4$ (idealization). 
So $R$ is a one-dimensional local ring and hence is pr\`esimplifiable and very strongly atomic. Let $a=(0,\bar{1}) \in R$, so $a$ is very strongly irreducible and prime as an element of $R$. So as an element of $R[X]$, $a$ is irreducible and prime. However, $a$ is not strongly irreducible as an element of $R[X]$, in fact $a$ is not even a product of strongly irreducible elements of $R[X]$. Hence a strongly irreducible (resp., $m$-irreducible, very strongly irreducible) of $R$ need not be strongly irreducible (resp., $m$-irreducible, very strongly irreducible) as an element of $R[X]$. For let $f=(1,\bar{0})+(2,\bar{0})X \in R[X]$; so $a=af^2=(af)f$. However, it is easily checked that $a \not\approx af$ and $a \not\approx f$; so $a$ is not strongly irreducible as an element of $R[X]$. Moreover, $a$ is not even a product of strongly irreducible elements of $R[X]$. For if $a=f_1 \cdots f_n$ where each $f_i \in R[X]$ is strongly irreducible, then $a$ irreducible in $R[X]$ gives that $a \sim f_i$ for some $i$. But then $f_i$ strongly irreducible gives that $a$ is strongly irreducible, a contradiction. Thus $R$ is very strongly atomic, but $R[X]$ is not even strongly atomic. However, since $R[X]$ is Noetherian, it is atomic. Note that $a \sim af$, but $a \not\approx af$ and $a \cong_R a$, but $a \not\cong_{R[X]} a$. Also, $R$ is pr\`esimplifiable, but $R[X]$ is not. 
\end{ex}

Now in general an element can be $m$-irreducible, but not very strongly irreducible (e.g., $a=(\bar{0},\bar{1})$ in $R=\mathbb{Z}_2 \times \mathbb{Z}_2$).  Note that the element $a=(\bar{0},\bar{1}) \in \mathbb{Z}_2 \times \mathbb{Z}_2$ just defined is strongly irreducible as an element of $R[X]$ but is not $m$-irreducible. (In fact, it is not even a product of $m$-irreducible elements of $R[X]$, here $R[X]$ is strongly atomic but not $m$-atomic.) However, we next show that for a nonzero element in a polynomial ring, the notions of $m$-irreducible and very strongly irreducible coincide. We have already given an example of an irreducible element of $R[X]$, $R=\mathbb{Z}_{(2)}(+)\mathbb{Z}_4$, that is not strongly irreducible.

\begin{thm} 
Let $R$ be a commutative ring and $X$ an indeterminate over $R$. Let $0 \neq f \in R[X]$.
\begin{enumerate}
\item $f$ is $m$-irreducible if and only if $f$ is very strongly irreducible.
\item $R[X]$ is very strongly atomic if and only if $R[X]$ is $m$-atomic.
\item Suppose that $R$ is reduced. If $f$ is very strongly irreducible, $f$ is regular.
\item Suppose that $R[X]$ is reduced and very strongly atomic. Then $R$ is an integral domain and hence $R$ is (very strongly) atomic.
\end{enumerate}
\end{thm}

\begin{proof}
\
\begin{enumerate}
\item A nonzero very strongly irreducible element is always $m$-irreducible. Conversely, suppose that $0 \neq f \in R[X]$ is $m$-irreducible. By Theorem 2.1 (6) either $f$ is very strongly irreducible or $R[X]f$ is an idempotent maximal ideal of $R[X]$. So in the second case we have $f=f^2g$ for some $g \in R[X]$. Then $e=fg$ is idempotent and $R[X]f=R[X]e$. Now $e \in R$, so $R[X]f=R[X]e \subsetneq (X,e) \subsetneq R[X]$, a contradiction. 
\item Clearly follows from (1).
\item Suppose that $R$ is reduced and $0 \neq f \in R[X]$ is very strongly atomic. By Theorem 2.1 (5), $ann(f) \subseteq J(R[X])=(nil(R))[X]=0$. So $f$ is regular. 
\item By (3) every very strongly irreducible element of $R[X]$ is regular. Since every nonzero nonunit of $R[X]$ is a product of very strongly irreducible elements, each nonzero element of $R[X]$ is regular, that is, $R$ is an integral domain. But it is easily checked that for $R$ an integral domain $R[X]$ very strongly atomic (=atomic) implies that $R$ is very strongly atomic (=atomic).
\end{enumerate} 
\end{proof}

It is well known that $R$ atomic does not imply $R[X]$ is atomic, even for $R$ an integral domain \cite{R}. However, it is easily seen that if $R$ is an integral domain and $R[X]$ is atomic, then $R$ itself is atomic. Coykendall and Trentham \cite{CT} give an example of a zero-dimensional quasilocal ring $S$ having no atoms with $S[X]$ being very strongly atomic. (Note that in \cite{CT} the term ``strongly atomic" is used for what we have called very strongly atomic. However, in this case since $S[X]$ is pr\`esimplifiable the various forms of atomicity all coincide.) They also show that if $R$ is a reduced ring with $R[X]$ very strongly atomic, then $R$ is very strongly atomic. However, by Theorem 3.4 (4) such a ring $R$ is actually an integral domain, and hence trivially $R[X]$ atomic implies $R$ is atomic. If $R$ is an integral domain, $R[X]$ satisfies ACCP if and only if $R$ does. While for any commutative ring $R$, $R[X]$ satisfies ACCP implies $R$ satisfies ACCP, the converse is false \cite{HL}.

For an integral domain $R$ a polynomial $f \in R[X]$ is \emph{indecomposable} if it is not a product of two polynomials of positive degree. This is equivalent to $f=gh$, $g,h \in R[X]$, implies $g \in R$ or $h \in R$ which may be restated as $f=gh$ implies $g\approx_{R[X]}a$ or $h\approx_{R[X]}a$ for some $a \in R$. Certainly an irreducible polynomial is indecomposable, but $2X+4 \in \mathbb{Z}[X]$ is indecomposable but not irreducible. Any polynomial of degree at most one is indecomposable, and any polynomial is a product of indecomposable polynomials. (With this definition the zero polynomial is indecomposable.)

We would like to extend the definition of an indecomposable polynomial to commutative rings with zero divisors. In general, the condition that $f$ must not be a product of two polynomials of positive degree is very strong. For if $R$ has nonzero nilpotent elements, $R[X]$ has units of positive degree and hence if $a \in R$ is a factor of $f$, so is $au$ for any unit $u \in R[X]$. With this in mind we define $f \in R[X]$ to be \emph{indecomposable} if $f=gh$ for $g,h \in R[X]$ implies $g \approx_{R[X]} a$ or $h \approx_{R[X]} a$ for some $a \in R$. We next collect some facts about indecomposable polynomials.

\begin{thm} Let $R$ be a commutative ring and $X$ an indeterminate. Let $f \in R[X]$. 
\begin{enumerate}
\item $f$ very strongly irreducible $\implies$ $f$ is indecomposable.
\item For $R$ reduced, $f$ is indecomposable $\iff$ $f$ is not a product of two polynomials of positive degree.
\item $0$ is indecomposable $\iff$ $0$ is irreducible, or equivalently, prime.\\
\end{enumerate}
\end{thm}

\begin{proof}
\
\begin{enumerate}
\item Suppose that $f$ is very strongly irreducible. If $f=0$, this follows from (3), so suppose $f\neq 0$. Now $f=gh$ for $g,h \in R[X]$ implies $f$ or $g$ is a unit in $R[X]$ and hence is very strongly associated with $1$ in $R[X]$.
\item Suppose that $R$ is reduced. Then $U(R[X])=U(R) \subset R$. Thus $f=gh \implies g \approx_{R[X]} a$ or $h \approx_{R[X]} a$ for some $a \in R$ is equivalent to $g \in R$ or $h \in R$. 
\item $(\Longrightarrow)$ Now if $0$ is irreducible, equivalently prime, this gives $0=gh \implies g=0$ or $h=0$, so $g \approx_{R[X]} 0$ or $h \approx_{R[X]} 0$, so $0$ is indecomposable.\\
$(\Longleftarrow)$ Suppose that $0$ is not irreducible, so $0=ab$ where $a,b \in R \backslash \{0\}$. Then $0=(a+aX)(b+bX)$. Suppose that say $a+aX \approx_{R[X]} c$ where $c \in R$. Then $a+aX=cu$ where $u \in U(R[X])$, so we may take $u=u_0+u_1X$ where $u_0 \in U(R)$ and $u_1 \in nil(R)$. Then $a=cu_0$ and $a=cu_1$, so $a=au_{0}^{-1}u_1$. Hence $0=a(1-u_0^{-1}u_1)$ where $u_0^{-1}u_1 \in nil(R)$ and hence $1-u_0^{-1}u_1 \in U(R)$. But then $a=0$, a contradiction.
\end{enumerate}
\end{proof}


We next give some examples of ``bad behavior'' of indecomposable elements.

\begin{ex}
\

\begin{enumerate}
\item (R reduced but $X$ not indecomposable)
Let $R=\mathbb{Z}_6$ and $X$ an indeterminate over $R$. Then $X=(\bar{2}+\bar{3}X)(\bar{3}+\bar{2}X)$, so $X$ is not indecomposable. In Theorem 4.1 we will see that $X$ is indecomposable $\iff$ $X$ is irreducible $\iff$ $R$ is indecomposable.
\item (R reduced with a nonzero element of $R$ not indecomposable)
Let $S$ and $T$ be indeterminates over $\mathbb{Z}$ and let $R=\mathbb{Z}[S,T]/(3S,ST,2T)=\mathbb{Z}[s,t]$. Let $X$ be an indeterminate over $R$. Then in $R[X]$, $6=(sX+2)(tX+3)$. Here $R$ is reduced, so $6$ is not indecomposable.
\item (While very strongly irreducible $\implies$ indecomposable, irreducible $\centernot\implies$ indecomposable)
This is a continuation of Example 3.4. Let $R=\mathbb{Z}_{(2)} (+) \mathbb{Z}_{4}$ (idealization), $a=(0,\bar{1})$ and $f=(1,\bar{0})+(2,\bar{0})X$. So $a$ is prime and irreducible in $R[X]$. Now $a=af^2=(af)f$. Hence $(a)=(af)$ and so $af$ is also irreducible and prime in $R[X]$. We claim that $a$ and $af$ are not indecomposable. Now $a=(af)f$ and $af=(af)f^2$, so it suffices to show that $af \not\approx_{R[X]} \alpha$, $f \not\approx \alpha$ and $f^2 \not \approx_{R[X]} \alpha$ for any $\alpha \in R$. Suppose that $af \approx_{R[X]} (c,d) \in R$. So $(0,\bar{1}) + (0,\bar{2})X = af=(c,d)((a_0,b_0)+(a_1,b_1)X+\cdots+(a_n,b_n)X^n)$ where $(a_0,b_0) \in U(R)$ and $(a_i,b_i) \in nil(R)$ for $i \geq 1$. Now $(0,\bar{1})=(c,d)(a_0,b_0)$ so $c=0$ since $(a_0,b_0)$ is a unit. Thus $(a_1,b_1) \in nil(R)=0(+)M$ gives $(0,\bar{2})=(c,d)(a_1,b_1)=(0,\bar{0})$, a contradiction. Next suppose that $f \approx_{R[X]}(c,d) \in R$. So $(1,\bar{0})+(2,\bar{0})X=f=(c,d)((a_0,b_0)+(a_1,b_1)X+\cdots+(a_n,b_n)X^n)$ where $(a_0,b_0) \in U(R)$ and $(a_i,b_i) \in nil(R)$ for $i \geq 1$. Now $(1,\bar{0})=(c,d)(a_0,b_0)$ so $(c,d) \in U(R)$. But then $(2,\bar{0})=(c,d)(a_1,b_1) \in nil(R)$, a contradiction. A similar proof shows that $f^2 \not \approx_{R[X]} (c,d) \in R$. 
\end{enumerate}
\end{ex}

\


\section{Factorization of Powers of $X$}


In this section we investigate when $X^n$ is a product of irreducible elements and when this factorization is unique. We first show that $X$ is irreducible, or equivalently, indecomposable, if and only if $R$ is indecomposable.

\begin{thm}
Let $R$ be a commutative ring and $X$ an indeterminate over $R$.
\begin{enumerate}
\item $X$ is prime $\iff$ $R$ is an integral domain.
\item The following are equivalent.
	\begin{enumerate}
	\item $X$ is irreducible.
	\item $X$ is indecomposable.
	\item $R$ is indecomposable.
	\end{enumerate}
\end{enumerate}
\end{thm}

\begin{proof}
\
\begin{enumerate}
\item $X$ is prime $\iff$ $(X)$ is prime $\iff$ $R \cong R[X]/(X)$ is an integral domain.
\item $(a) \implies (b)$ Suppose that $X$ is irreducible. Since $X$ is regular, it is actually very strongly irreducible. By Theorem 3.6 (1), $X$ is indecomposable. $(b) \implies (c)$ Suppose that $X$ is indecomposable, but $R$ is not indecomposable. Let $e\neq 0,1$ be a nontrivial idempotent of $R$. Now $X=(e+(1-e)X)((1-e)+eX)$. Without loss of generality, we can assume that $e+(1-e)X\approx_{R[X]} a$ where $a\in R$. So $e+(1-e)X=a(b_0+b_1X+\cdots+b_nX^n)$ where $b_0+b_1X+\cdots+b_nX^n$ is a unit of $R[X]$. So $b_0$ is a unit of $R$. Then $e=ab_0$, so $a=eb_0^{-1}$. Hence $1-e=ab_1=eb_0^{-1}b_1$, so $1-e=(1-e)^2=(1-e)eb_0^{-1}b_1=0$, a contradiction. $(c) \implies (a)$ Suppose that $R$ is indecomposable. Let $X=fg$ where $f=a_0+a_1X+\cdots+a_nX^n$ and $g=b_0+b_1X+\cdots+b_mX^m$ are in $R[X]$. So $0=a_0b_0$ and $1=a_0b_1+a_1b_0$. Then $a_0=a_0\cdot 1=a_0(a_0b_1+a_1b_0)=a_0^2b_1+a_1a_0b_0=a_0^2b_1$. Put $e=a_0b_1$. So $e^2=e$ and $(a_0)=(e)$. Since $R$ is indecomposable, $e=0$ or $e=1$. If $e=0$, then $a_0=0$. So $X=X(a_1+a_2X+\cdots+a_nX^{n-1})g$. Thus $1=(a_1+a_2X+\cdots+a_nX^{n-1})g$ so $g$ is a unit. If $e=1$, $a_0$ is a unit. So $a_0b_0=0$ gives $b_0=0$. Then as before $f$ is a unit. 
\end{enumerate}
\end{proof}

Since the map $R[X] \to R[X]$ given by $X \to X-a$, $a \in R$, is an automorphism, $X$ is irreducible $\iff$ $X-a$ is irreducible. Thus if $R$ is indecomposable, each $X-a$ is irreducible.

We next generalize Theorem 4.1 to powers of $X$ being a product of atoms. We need the following lemma.

\begin{lem}
Let $R$ be a commutative ring and $X$ an indeterminate over $R$. Let $f \in R[X]$ be a nonunit factor of $X^n$, $n \geq 1$. Then $f$ is indecomposable if and only if $f$ is (very strongly) irreducible.
\end{lem}

\begin{proof}
$(\Longleftarrow)$ Let $f$ be an irreducible factor of $X^n$. Since $X^n$ is regular, so is $f$. Hence $f$ is very strongly irreducible. By Theorem 3.6, $f$ is indecomposable. \\
$(\Longrightarrow)$ Suppose that $f$ is a nonunit factor of $X^n$ that is indecomposable, say $X^n=fg$ where $g \in R[X]$. Let $f=f_1f_2$ where $f_i \in R[X]$. Since $f$ is indecomposable, say $f_1 \approx a \in R$, so $f_1=au$ where $u \in U(R[X])$. Now $X^n=au(f_2g)$. Since $a$ is a factor of $X^n$, it is regular. Write $f_2g=b_mX^m+\cdots+b_sX^s=X^m(b_m+\cdots+b_sX^{s-m})$ where $b_m \neq 0$. Since $au(0)$ is regular, $m=n$ and $1=au(0)b_m$. Thus $a$ is a unit and hence $f_1=au$ is a unit. So $f$ is very strongly irreducible.

\end{proof}

\begin{thm}
Let $R$ be a commutative ring and $X$ an indeterminate over $R$. Then the following are equivalent. 
\begin{enumerate}
\item $R$ is a finite direct product of indecomposable rings.
\item $X$ is finite product of atoms (resp., indecomposable elements).
\item Each $X^n$, $n \geq 1$, is a finite product of atoms (resp., indecomposable elements).
\item Some $X^n$, $n \geq 1$, is a finite product of atoms (resp., indecomposable elements).
\end{enumerate}
\end{thm}

In this case $X$ can be written uniquely up to order and unit associates as a finite product of atoms (resp., indecomposable elements). 

\begin{proof} We first do the ``atomic case". The ``indecomposable case" then follows from Lemma 4.2. $(1) \implies (2)$ Let $R=R_1 \times \cdots \times R_m$ where each $R_i$ is indecomposable. By Theorem 4.1, $X$ is an atom of $R_i[X]$. Now identifying $R[X] = R_1[X] \times \cdots \times R_m[X]$, $f_i=(1,1,\hdots,1,X,1,\hdots,1),$ where $X$ is in the $i^{th}$ coordinate, is an atom of $R_1[X] \times \cdots \times R_m[X]$. But then $X=f_1 \cdots f_m$ is a product of atoms. $(2) \implies (3) \implies (4)$ Clear. \noindent $(4) \implies (1)$ Let $R=R_1 \times \cdots \times R_s$. Suppose that $X^n=f_1 \cdots f_m$ where each $f_i \in R[X]$ is irreducible. Since $f_i$ is irreducible in $R_1[X] \times \cdots \times R_s[X]$, it has exactly one coordinate that is a nonunit. Since each coordinate of $X^n$ is a nonunit, $s \leq m$. Since there is a bound on the number of factors in a direct product decomposition of $R$, $R$ is a finite direct product of indecomposable rings. (The equivalence $(1) \iff (2)$ for the atomic case is given in \citep[Theorem 6.4]{AV1}.)


Suppose that $R=R_1 \times \cdots \times R_m$ where $R_i$ is indecomposable. Then in $R[X]=R_1[X] \times \cdots \times R_m[X]$, $X=(X,\hdots,X)=(X,1,\hdots)(1,X,1,\hdots)\cdots(1,\hdots,1,X)$ is a product of $m$ atoms. Suppose that $X=f_1 \cdots f_s$ where $f_i$ is irreducible in $R[X]$. Let $f_i=(f_{i1},\hdots,f_{im})$ where $f_{ij} \in R_j[X]$. So for each $i$ exactly one $f_{ij}$, say $f_{{ij}_{i}}$, is not a unit and it is an atom in $R_{j_i}[X]$. Now in $R_{j_i}[X]$, $X=f_{1{j_i}}\cdots f_{i{j_i}}\cdots f_{s{j_i}}$. Since $X$ is an atom of $R_{j_i}[X]$, as $R_{j_i}$ is indecomposable, $f_{{ij}_{i}}=u_{j_i}X$ where $u_{j_i} \in U(R_{j_i}[X])$. Now $f_i=(f_{i1},\hdots,f_{{ij}_{i-1}}, u_{j_i}X,f_{{ij}_{i+1}},\hdots,f_{im})=(f_{i1},\hdots,f_{{ij}_{i-1}}, u_{j_i},f_{{ij}_{i+1}},\hdots,f_{im})(1,\hdots,1,X,1,\hdots,1)$ where the first factor is a unit of $R[X]$. Since $f_1 \cdots f_s=(X,\hdots,X)$ we must have $s=m$ and $\{j_1,\hdots,j_m\}=\{1,2,\hdots,m\}$. Thus the factorization $X=(X,\hdots,X)=(X,1,\hdots,1)\cdots(1,\hdots,1,X)$ is unique up to order and unit multiplication.

\end{proof}

\begin{cor} 
Let $R$ be a commutative ring and $X$ an indeterminate over $R$. Then the following are equivalent.
\begin{enumerate}
\item $R$ is a finite direct product of integral domains.
\item $X$ is a product of primes.
\item Each $X^n$, $n\geq 1$, is a product of primes.
\item For some $n\geq 1$, $X^n$ is a product of primes.
\end{enumerate}
\end{cor}

\begin{proof}
$(1) \Longrightarrow (2)$ Let $R= R_1 \times \cdots \times R_m$ where $R_i$ is an integral domain. Then $X$ is prime in $R_i[X]$ so $(1,\hdots,1,X,1,\hdots,1)$ is prime in $R[X]=R_1[X] \times \cdots \times R_m[X]$. So $X=(X,\hdots,X)=(X,1,\hdots,1)\cdots(1,\hdots,1,X)$ is a product of primes. $(2) \Longrightarrow (3) \Longrightarrow (4)$ Clear. $(4) \Longrightarrow (1)$. Suppose that $X^n$ is a product of primes. Then $X^n$ is a product of atoms, so $R=R_1 \times \cdots \times R_m$ where $R_i$ is indecomposable. So $X$ is irreducible in $R_i[X]$ and $X^n$ is a product of primes in $R_i[X]$, say $X^n=p_1 \cdots p_s$. Now $p_i$ prime gives $p_i \mid X$ and since $X$ is irreducible, $p_i \sim X$. Thus $X$ is prime in $R_i[X]$ and hence $R_i$ is an integral domain.

\end{proof}


We next determine when each $X^n$ has a unique factorization into atoms.

\begin{thm}
Let $R$ be a commutative ring and $X$ an indeterminate over $R$. Then the following are equivalent.
\begin{enumerate}
\item $R$ is reduced and is a finite direct product of indecomposable rings. 
\item For each $n \geq 1$, $X^n$ has a unique factorization into irreducibles up to order and unit multiplication.
\item For some $n \geq 2$, $X^n$ has unique factorization into irreducibles up to order and unit multiplication.
\item $X^2$ has unique factorization into irreducibles up to order and unit multiplication.
\end{enumerate}
\end{thm}


\begin{proof}
$(2) \implies (3)$ Clear. $(3) \implies (4)$ Suppose that $X^n$ has a unique factorization into irreducibles. By Theorem 4.3, $X^2$ is a finite product of irreducibles. Suppose that $X^2$ has two different factorizations into irreducibles. Now either $n=2$ or $X^{n-2}$ is a product of irreducibles. Thus $X^n=X^2X^{n-2}$ also has two different factorizations into irreducibles, a contradiction. $(4) \implies (1)$ By Theorem 4.3, $R$ is a finite direct product of indecomposable rings, say $R=R_1 \times \cdots \times R_m$. Suppose that $R$ is not reduced. Hence some $R_i$ is not reduced. Let $0 \neq a \in R_i$ with $a^2=0$. Then in $R_i[X]$, $X^2=X \cdot X=(X+a)(X-a)$ are two different atomic factorizations of $X^2$. But this leads to two different atomic factorizations of $X^2$ in $R[X]$, a contradiction. $(1) \implies (2)$ Since $R$ is a finite direct product of indecomposable rings, $X^n$ is a product of atoms. Since $R$ is reduced, $U(R[X])=U(R)$. Suppose that $R=R_1 \times \cdots \times R_m$ where each $R_i$ is indecomposable. An easy modification of the proof in Theorem 4.3 that $X$ has a unique factorizations shows that if $X^n$ has a unique factorization into atoms in each $R_i[X]$, then $X^n$ has a unique factorization into atoms in $R[X]$. Thus we may assume that $R$ is a reduced indecomposable ring. Since $X$ has a unique factorization into atoms by Theorem 4.3, we may assume that $n \geq 2$. Let $f=a_0+a_1X+\cdots+a_sX^s$ be an irreducible factor of $X^n$, say $X^n=fg$ where $g=b_0+b_1X+\cdots+b_tX^t$. Then $0=a_0b_0=a_0b_1+a_1b_0=\cdots=a_0b_{n-1}+a_1b_{n-2}+\cdots+a_{n-1}b_0$ and $a_0b_n+a_1b_{n-1}+\cdots+a_nb_0=1$. Now suppose that we have shown that $0=a_0b_0=a_0b_1=\cdots=a_0b_i$ where $0 \leq i \leq n-2$. Then $0=a_0\cdot 0=a_0(a_0b_{i+1}+a_1b_i+\cdots a_{i+1}b_0)=a_0^2b_{i+1}+a_1a_0b_i+\cdots+a_{i+1}a_0b_0=a_0^2b_{i+1}$. Then $(a_0b_{i+1})^2=0$, so $a_0b_{i+1}=0$ since $R$ is reduced. Hence $0=a_0b_0=a_0b_1=\cdots=a_0b_{n-1}$. So $a_0=a_0\cdot 1=a_0(a_0b_n+a_1b_{n-1}+\cdots+a_nb_0)=a_0^2b_n+a_1a_0b_{n-1}+\cdots+a_na_0b_0=a_0^2b_n$. So $e=a_0b_n$ is idempotent with $(a_0)=(e)$. Since $R$ is indecomposable, either $e=1$, so $a_0$ is a unit or $e=0$ so $a_0=0$. First suppose that $a_0$ is a unit. Let $P$ be a prime ideal of $R$. Then in $(R/P)[X]$, $\bar{a}_0+\bar{a}_1X+\cdots+\bar{a}_sX^s$ is a factor of $X^n$. Since $a_0$ is a unit, $\bar{a}_0 \neq 0$, so $\bar{a}_1=\cdots=\bar{a}_s=\bar{0}$, that is, $a_1,\hdots,a_s \in P$. Then for $i \geq 1$, $a_i \in \bigcap \{P \mid P \in Spec(R)\}=nil(R)=0$. So $f=a_0$ is a unit, a contradiction. So $e=0$ and hence $a_0=0$. So $f=a_1X+\cdots+a_sX^s=(a_1+\hdots +a_sX^{s-1})X$. Now $f$ is irreducible and regular and hence is very strongly irreducible. Thus $a_1+\hdots+a_sX^{s-1} \in U(R[X])=U(R)$. So $a_1 \in U(R)$ and $0=a_2=\cdots=a_s$. So $f=a_1X$ where $a_1 \in U(R)$. 

\end{proof}

\begin{cor}
Let $R$ be a commutative ring and $X$ an indeterminate over $R$. Then the following conditions are equivalent.
\begin{enumerate}
\item $R$ is reduced and indecomposable.
\item $S=\{uX^n \mid u \in U(R), n \geq 1\}$ is a saturated multiplicatively closed subset of $R[X]$. 
\item For each $n \geq 1$, the only (irreducible) factors of $X^n$ are $uX^m$ where $u \in U(R)$ and $0 \leq m \leq n$ $(m=1)$.
\item For some $n \geq 2$, the only irreducible factors of $X^n$ are $uX^m$ where $u \in U(R)$ and $0 \leq m \leq n$ $(m=1)$.
\item The only (irreducible) factors of $X^2$ are $uX^m$ where $u \in U(R)$ and $0 \leq m \leq 2$ $(m=1)$.
\end{enumerate}
\end{cor}


It is easy to see that in $(2)-(5)$ of Corollary 4.6 we can replace $U(R)$ by $U(R[X])$. Also let us mention a related result of Gilmer and Heinzer \citep[Corollary 7]{GH}: $R$ is reduced and indecomposable if and only if the set of polynomials of $R[X]$ with unit leading coefficient is a saturated multiplicatively closed set.

We have seen that each power of $X^n$ has unique factorization into irreducibles if and only if $R$ is reduced and a finite direct product of indecomposable rings. The smallest example of a ring in which each $X^n$ is a product of irreducibles but the factorization is not unique is $\mathbb{Z}_4.$ Here $\mathbb{Z}_4[X]$ is pr\`esimplifiable so all the forms of irreducibility coincide. Also, $\mathbb{Z}_4$ is indecomposable, so $X$ is strongly irreducible and hence indecomposable. So each $aX+b$ where $a \in \{\bar{1},\bar{3}\}$ and $b \in \mathbb{Z}_4$ is irreducible and indecomposable. Now $\bar{2}$ is prime and hence irreducible and indecomposable. It is easily checked that $\bar{2}X$ and $\bar{2}X+\bar{2}$ are indecomposable. The only atomic factorizations are $\bar{2}X=\bar{2}\cdot X=\bar{2}(X+\bar{2})$ and $\bar{2}X+\bar{2}=\bar{2}(X+\bar{1})=\bar{2}(X+\bar{3})$. Now for $n\geq 1$, $X^n+\bar{2}$ is irreducible. For each $n\geq 2$, $X^n=X^{n-2}(X+\bar{2})^2$ are two different atomic factorizations of $X^n$. Now for $n \geq 1$, $X^{2n}=(X^n+\bar{2})^2$, so $X^{2n}$ has atomic factorizations of length $2$ and $2n$. It is easily checked that any atomic factorization of $X^{2n}$ has length $l$ with $2 \leq l \leq 2n$. Likewise, for $n\geq 1$, $X^{2n+1}=X(X^n+\bar{2})^2$ and it is easily checked that any atomic factorization of $X^{2n+1}$ has length $l$ where $3 \leq l \leq 2n+1$. For $n \geq 1$, let $L(X^n)=\{m \mid f_1 \cdots f_m$ where $f_i$ is irreducible$\},$ the set of lengths of $X^n$. It is easily checked that $L(X)=\{1\}$, $L(X^2)=\{2\}$, and $L(X^3)=\{3\}$. Belshoff, Kline, and Rogers \cite{BKR} have shown that $L(X^4)=\{2,4\}$, $L(X^5)=\{3,4\}$, $L(X^6)=\{2,4,6\}$ and for $n \geq 7$ $L(X^n)=\{2,3,\hdots,n-4\} \cup \{n-2,n\}$ for $n$ even and $L(X^n)=\{3,4,\hdots,n-4\} \cup \{n-2,n\}$ for $n$ odd. In fact, they have determined $L(X^n)$ where $(R,M)$ is a local Artinian ring with $M^2=0$. 


\section{Unique Factorization in $R[X]$}


It is well known that for $D$ an integral domain, $D[X]$ is a UFD if and only if $D$ is a UFD. There are a number of ways in which the notion of a UFD can be extended to commutative rings with zero divisors. In this section we investigate when $R[X]$, $R$ a commutative ring, satisfies any of these various generalizations of a UFD. While structure theories for certain of these generalizations are known, we strive to derive our results from ``first principles" using certain features of $R[X]$. 

Of course an integral domain $D$ is a UFD if (1) every nonzero nonunit of $D$ is a product of irreducibles, and (2) this factorization into irreducibles is unique up to order and associates. In the nondomain case we have a number of ways to define ``associate" and ``irreducible." We begin with the following definition:


\begin{defn}
Let $R$ be a commutative ring and $a \in R$ a nonunit. Two factorizations of $a$ into nonunits $a=a_1 \cdots a_n=b_1 \cdots b_m$ are \emph{isomorphic} (resp., \emph{strongly isomorphic}, \emph{very strongly isomorphic}) if $n=m$ and there exists a permutation $\sigma \in S_n$ with $a_i \sim b_{\sigma(i)}$ (resp., $a_i \approx b_{\sigma(i)}$, $a_i \cong b_{\sigma(i)}$). Two factorizations of $a$ into nonunits $a=a_1 \cdots a_n=b_1 \cdots b_m$ are  \emph{homomorphic} (resp., \emph{strongly homomorphic}, \emph{very strongly homomorphic}, \emph{weakly homomorphic}) if for each $i \in \{1,\hdots,n\}$ there exists a $j \in \{1,\hdots,m\}$ with $a_i \sim b_j$ (resp., $a_i \approx b_j, a_i \cong b_j, a_i \mid b_j$) and for each $i \in \{1,\hdots,m\}$ there exists a $j \in \{1,\hdots,n\}$ with $b_i \sim a_j$ (resp., $b_i \approx a_j$, $b_i \cong a_j$, $b_i \mid a_j$).
\end{defn}

Note that isomorphic factorizations are homomorphic, but the converse may be false even if $a$ is regular. For consider the two homomorphic factorizations of $X^{15}$ into irreducibles in $\mathbb{Q}[X^2,X^3]$: $$X^{15}=X^2 \cdot X^2 \cdot X^2 \cdot X^2 \cdot X^2 \cdot X^2 \cdot X^3 = X^2 \cdot X^2 \cdot X^2 \cdot X^3 \cdot X^3 \cdot X^3.$$ Also, if $e$ is a nontrivial idempotent of a commutative ring $R$, $e=e^2$ are homomorphic but not isomorphic factorizations of $e$. However, if for each nonunit regular element of $R$ any two atomic factorizations are weakly homomorphic, then any two atomic factorizations of a regular nonunit are isomorphic.

Each form of atomicity and ``isomorphic" leads to a type of unique factorization ring as given by our next definition.

\begin{defn}
Let $R$ be a commutative ring. Let $\alpha \in \{$atomic, strongly atomic, $m$-atomic, very strongly atomic$\}$ and $\beta \in \{$isomorphic, strongly isomorphic, very strongly isomorphic$\}$. Then $R$ is an $(\alpha,\beta)$-\emph{unique factorization ring} if (1) $R$ is $\alpha$ and $(2)$ any two factorizations of a nonzero, nonunit element into irreducible elements of the type used to define $\alpha$ are $\beta$. 
\end{defn}


Note that for any choice of $\alpha$ and $\beta$, an $(\alpha,\beta)$-unique factorization ring $R$ is pr\`esimplifiable. Thus in an $(\alpha,\beta)$-unique factorization ring, the notions of associate, strongly associate, and very strongly associate coincide and hence the notions of irreducible, strongly irreducible, $m$-irreducible, and very strongly irreducible coincide as do the notions of isomorphic, strongly isomorphic, and very strongly isomorphic factorizations.

\begin{defn}
Let $R$ be a commutative ring. Then $R$ is a \emph{unique factorization ring (UFR)} if $R$ is an $(\alpha,\beta)$-unique factorization ring for some (and hence all) $(\alpha,\beta)$. 
\end{defn}

In our terminology Bouvier \cite{B2} showed that $R$ is an ($m$-atomic, isomorphic)-unique factorization ring if and only if $R$ is either (1) a UFD, (2) an SPIR, or (3) a quasilocal ring $(R,M)$ with $M^2$ while Galovich \cite{G} gave a similar characterization of (very strongly atomic, strongly isomorphic)-unique factorization rings. So $R$ is a UFR if and only if $R$ is either (1) a UFD, (2) an SPIR, or (3) a quasilocal ring $(R,M)$ with $M^2=0$. Since a polynomial ring $R[X]$ is never quasilocal, it follows that $R[X]$ is a UFR if and only if $R[X]$ (or equivalently, $R$) is a UFD. However, we give a simple proof of this result without the use of the previously mentioned structure theory for UFRs, see Corollary 5.5.


In \citep[Theorem 4.6]{AV1} it was shown that for a commutative ring $R$ the following conditions are equivalent: (1) $R$ is either (a) a UFD, (b) a quasilocal ring $(R,M)$ with $M^2=0$, or (c) a finite direct product of SPIRs and fields, (2) $R$ is atomic and any two factorizations of a nonzero, nonunit element into irreducibles are homomorphic, and (3) $R$ is $m$-atomic and any two factorizations of a nonzero, nonunit element into $m$-irreducibles are strongly homomorphic. And it was noted that this result could be stated for any form of atomicity except for very strongly atomic and for either homomorphic or strongly homomorphic. In the statement of this result we cannot replace ``atomic" by ``very strongly atomic" since a direct product of two or more rings where at least one is an integral domain is not very strongly atomic. Using \citep[Theorem 4.6]{AV1} one can show that the following are equivalent for a commutative ring $R$: (1) $R$ is either (a) a UFD, (b) a quasilocal ring $(R,M)$ with $M^2=0$, or (c) a finite direct product of SPIRs that are not fields and (2) $R$ is very strongly atomic and any two factorizations of a nonzero, nonunit element into very strongly irreducible elements are homomorphic (or equivalently strongly homomorphic, or very strongly homomorphic.)


Now since $R[X]$ has infinitely many maximal ideals using the previously mentioned structure theory, we have that the following are equivalent: (1) $R[X]$ is a UFD, (2) $R[X]$ is atomic and any two factorizations of a nonzero, nonunit elements into irreducibles are homomorphic, and (3) $R[X]$ is very strongly atomic and any two factorizations of a nonzero, nonunit element into very strongly irreducible elements are very strongly homomorphic. However, we will again give a simple proof from first principles, see Corollary 5.5.

Another generalization of a UFD was given in \cite{UFRZD1}. A commutative ring $R$ was defined to be a \emph{weak UFR} if (1) $R$ is atomic and (2) any two factorizations of a nonzero nonunit of $R$ into irreducibles are weakly homomorphic. It was shown that the following conditions are equivalent: (1) $R$ is a weak UFR, (2) every nonzero nonunit of $R$ is a finite product of weakly prime elements, (3) $R$ is atomic and every irreducible elements of $R$ is weakly prime, and (4) $R$ is either a finite direct product of UFDs and SPIRs or $(R,M)$ is a quasilocal ring with $M^2=0$. (\citep[Theorem 2.3]{UFRZD1} gives that (1)-(3) are equivalent while \citep[Theorem 2.13]{UFRZD1} gives (1) and (4) are equivalent.) Thus $R[X]$ is a weak UFR if and only if $R$ is a finite direct product of UFDs. In Theorem 5.4 we give a proof of this without recourse to the structure theorem given in \cite{UFRZD1}.

Fletcher \cite{UFR} defined a ``unique factorization ring" in yet another way. Let $R$ be a commutative ring. For $r \in R$, let $U(r)=\{s \in R \mid s(r)=(r)\}$. He defined the \emph{U-decomposition} of a nonunit $a \in R$ as $a=a_1 \cdots a_k \lceil b_1 \cdots b_n \rceil$ where $a_i, b_j$ are irreducible, $a_i \in U(b_1 \cdots b_n)$ for each $i=1,\hdots,k$ and $b_j \not\in U(b_1 \cdots \hat{b_j} \cdots b_n)$ for each $j=1,\hdots,n$. He then called $R$ a ``unique factorization ring" (which we will called a \emph{Fletcher unique factorization ring}) if (1) every nonunit of $R$ has a $U$-decomposition, and (2) if $a_1 \cdots a_k \lceil b_1 \cdots b_n \rceil = a'_1 \cdots a'_{k'} \lceil b'_1 \cdots b'_{n'}\rceil$ are two $U$-decompositions of a nonunit element of $R$, then $n=n'$ and after a reordering, if necessary, $b_i \sim b_i'$ for $i=1,\hdots,n$. As any atomic factorization of an element can be ``refined" to a $U$-decomposition it is (2) that is essential. Note that if $a \in R$ is a regular nonunit, then a $U$-decomposition of $a$ is just a factorization of $a$ into irreducible elements. So in a Fletcher unique factorization ring any two factorizations of a regular nonunit into irreducible elements are isomorphic. We call $R$ a \emph{factorial ring} if (1) every regular nonunit of $R$ is a finite product of irreducible elements and (2) any two factorizations of a regular nonunit into irreducible elements are isomorphic. Thus a Fletcher unique factorization ring is a factorial ring. In \cite{SUFR} Fletcher proved that $R$ is a Fletcher unique factorization ring if and only if $R$ is a finite direct product of UFDs and SPIRs. Thus $R[X]$ is a Fletcher unique factorization ring if and only if $R$ is a finite product of UFDs. We prove this without recourse to Fletchers's structure theory in our next theorem. It is worth noting \citep[Theorem 2.1]{UFRZD1} that $R$ is a Fletcher unique factorization ring if and only if (1) $R$ is atomic and (2) any two atomic factorizations of a nonunit (possibly 0) are weakly homomorphic.


\begin{thm} Let $R$ be a commutative ring. Then the following conditions are equivalent. 
\begin{enumerate}
\item $R$ is a finite direct product of UFDs (possibly fields).
\item Every (nonzero) nonunit of $R[X]$ is a product of prime elements.
\item Every regular nonunit of $R[X]$ is a product of prime elements.
\item $R[X]$ is factorial, that is, every regular nonunit of $R[X]$ is a product of irreducible elements and any two factorizations of a regular nonunit into irreducible elements are isomorphic.
\item Every regular nonunit of $R[X]$ is a product of irreducible elements and any two factorizations of a regular nonunit into irreducible elements are homomorphic.
\item $R[X]$ is a Fletcher unique factorization ring.
\item $R[X]$ is a weak UFR.
\end{enumerate}
\end{thm}

\begin{proof}
$(1) \implies (2)$ Suppose that $R=R_1 \times \cdots \times R_n$ where each $R_i$ is a UFD. Then each $R_i[X]$ is a UFD and $R[X]=R_1[X] \times \cdots \times R_n[X]$. But it is easily checked that in a direct product of UFDs, each nonunit is a product of prime elements. $(2) \implies (3)$ Clear. $(3) \implies (4)$ Now a prime element is irreducible, so every regular nonunit of $R[X]$ is a product of irreducible elements. Moreover, since a regular irreducible element is a product of primes, it is itself prime. But it is well known that the factorization of a regular element into primes is unique up to order and associates. $(4) \implies (5)$ Clear.
$(5) \implies (1)$ Since $X$ is a product of irreducible elements, $R=R_1 \times \cdots \times R_n$ where each $R_i$ is indecomposable. Now $R[X]=R_1[X] \times \cdots \times R_n[X]$ and it is easily checked that each $R_i[X]$ satisfies (5). Thus we can assume that $R$ is indecomposable and we must prove that $R$ is an integral domain (for then $R$ is a UFD). Suppose that there exist nonzero elements $a$ and $b$ of $R$ with $ab=0$. Since $R$ is indecomposable, $X, X-a, X-b$, and $X-(a+b)$ are irreducible, so $(X-a)(X-b)=X^2-(a+b)X=X(X-(a+b))$ are two factorizations of the regular element $X^2-(a+b)X$ into irreducibles that are not homomorphic since $X-c$ and $X-d$ are associates if and only if $c=d$. $(6) \implies (4)$ In the paragraph preceding this theorem we remarked that a Fletcher unique factorization ring is a factorial ring. $(1) \implies (6)$ $R[X]$ is a finite direct product of UFDs. Now certainly a UFD is a Fletcher unique factorization ring and it is easily checked that a direct product of Fletcher unique factorization rings is a Fletcher unique factorization ring. Thus $R[X]$ is a Fletcher unique factorization ring.
$(1) \implies (7)$ Here $R[X]$ is a finite direct product of UFDs and it is easily checked that a finite direct product of UFDs is a weak UFR. $(7) \implies (5)$ If for every regular nonunit any two atomic factorizations are weakly homomorphic, they are actually homomorphic. 

\end{proof}

The equivalence of (1), (4), and (6) of Theorem 5.4 is given in \citep[Theorem 3.8]{UFRZD} (also, see \cite{UFRZDC}) where we use the term UFR for a Fletcher unique factorization ring. The proof given there involves Krull rings.


\begin{cor}
For a commutative ring $R$ the following conditions are equivalent.
\begin{enumerate}
\item $R[X]$ (or equivalently $R$) is a UFD.
\item $R[X]$ is a unique factorization ring.
\item $R[X]$ is atomic and any two factorizations of nonzero nonunits into irreducibles are homomorphic.
\end{enumerate}
\end{cor}

\begin{proof}

$(1) \implies (2) \implies (3)$ Clear. $(3) \implies (1)$ By $(5) \implies (1)$ of Theorem 5.4, $R=D_1 \times \cdots \times D_n$ where each $D_i$ is a UFD. Suppose that $n>1$. Then $(0,1,\hdots,1)$ and $(X,1,\hdots,1)$ are irreducible elements of $R[X]=D_1[X] \times \cdots \times D_n[X]$. But $(0,1,\hdots,1)=(0,1,\hdots,1)(X,1,\hdots,1)$ are two nonhomomorphic factorizations of $(0,1,\hdots,1)$ into irreducibles.

\end{proof}


We next discuss a theory of factorization introduced in \cite{CAV}. Let $R$ be a commutative ring. By a \emph{$\mu$-factorization} of a nonunit $a \in R$ we mean a factorization $a=\lambda a_1 \cdots a_n$ where $\lambda \in U(R)$ and each $a_i$ is a nonunit. A \emph{factorization} of $a$ is a $\mu$-factorization  with $\lambda=1$ (which is then omitted). The $\mu$-factorization $\lambda a_1 \cdots a_n$ is \emph{(strongly) reduced} if $\lambda a_1 \cdots a_n \neq \lambda a_1 \cdots \hat{a_i} \cdots a_n$ ($\lambda a_1 \cdots a_n \neq \lambda a_{i_1} \cdots a_{i_j}$ for any proper subset $\{i_1,\hdots,i_j\}$ of $\{1,\hdots,n\}$) and is \emph{(strongly) $\mu$-reduced} if $\lambda a_1 \cdots a_n \neq \lambda' a_1 \cdots \hat{a_i} \cdots a_n \hspace{1mm} (\lambda a_1 \cdots a_n \neq \lambda' a_{i_1} \cdots a_{i_j})$ for any proper subset $\{i_1,\hdots, i_j\}$ of $\{1,\hdots,n\})$. So we have

\

\vspace{-5mm}
\[
\xymatrix@C=0.45cm@1{ 
 &\text{strongly $\mu$-reduced } \ar@{=>}[r] \ar@{=>}[d] & \text{ strongly reduced} \ar@{=>}[d]\\
 &\text{ $\mu$-reduced } \ar@{=>}[r] &\text{ reduced}}\\
\] 

\

\noindent and it is easily seen that none of these implications can be reversed. Note that if $a \in R$ is a regular nonunit, then any $\lambda$-factorization is strongly $\mu$-reduced. Also, it is easily seen that any $(\mu$-) factorization can be ``reduced" to a strongly $(\mu$-) factorization. Then $R$ is a \emph{(weak) [strongly] $\mu$-reduced unique factorization ring} if (1) $R$ is atomic and (2) for each (nonzero) nonunit $a \in R$, if $a=\lambda_1 a_1 \cdots a_n = \lambda_2 b_1 \cdots b_m$ are two [strongly] $\mu$-reduced atomic factorizations of $a$, then $n=m$, and after re-ordering, if necessary, $a_i \sim b_i$. And $R$ is a \emph{(weak) [strongly] reduced unique factorization ring} if (1) $R$ is atomic and (2) if $a \in R$ is a (nonzero) nonunit of $R$ and $a=\lambda_1 a_1 \cdots a_n = \lambda_2 b_1 \cdots b_m$ are two atomic [strongly] reduced factorizations of $a$, then $n=m$ and after re-ordering, if necessary, $a_i \sim b_i$. (Note that in (2) there is no loss in generality in just taking $\lambda=1$ and hence omitting it.)

Thus we have

\

\vspace{-5mm}
\[
\xymatrix@C=0.45cm@1{ 
 &R \text{ strongly $\mu$-reduced } \ar@{<=}[r] \ar@{<=}[d] & R \text{ strongly reduced} \ar@{<=}[d]\\
 &R \text{ $\mu$-reduced } \ar@{<=}[r] &R \text{ reduced}}\\
\] 

\

\noindent and in the next paragraph we note that the two vertical implications are actually equivalences. 


Theorem 3.3 \cite{CAV} gave that $R$ is a (strongly) $\mu$-reduced UFR if and only if $R$ is a finite direct product of UFDs and SPIRs while Theorem 3.4 \cite{CAV} gave that $R$ is a (strongly) reduced UFR if and only if $R$ is a UFD, SPIR, or a finite direct product $R=D_1 \times \cdots \times D_n$ where each $D_i$ is a UFD with $U(D_i)=\{1\}$. So in the case of a polynomial ring $R[X]$ we have that $R[X]$ is a (strongly) $\mu$-reduced UFR if and only if $R$ is a finite direct product of UFDs and $R[X]$ is a (strongly) reduced UFR if and only  if $R$ is a UFD or $R=D_1 \times \cdots \times D_n$ where $D_i$ is a UFD with $U(D_i)=\{1\}$. We next give simple proofs of these results for $R[X]$ using Theorem 5.4.

\begin{thm}
For a commutative ring $R$, the following conditions are equivalent. 
\begin{enumerate}
\item $R[X]$ is a (weak) strongly $\mu$-reduced UFR.
\item $R[X]$ is a (weak) $\mu$-reduced UFR.
\item $R$ is a finite direct product of UFDs.
\end{enumerate}
\end{thm}  

\begin{proof}
$(3) \implies (2)$ $R[X]$ is a finite direct product of UFDs and it is easily checked that a finite direct product of UFDs is a $\mu$-reduced UFR. $(2) \implies (1)$ Clear. $(1) \implies (3)$ Assume that $R[X]$ is a weak strongly $\mu$-reduced UFR. Then for a regular nonunit $f \in R[X]$ any two atomic factorizations of $f$ are isomorphic, that is, $R[X]$ is factorial. By Theorem 5.4, $R$ is a finite direct product of UFDs.

\end{proof}


\begin{thm}
For a commutative ring $R$, the following conditions are equivalent.
\begin{enumerate}
\item $R[X]$ is a (weak) strongly reduced UFR.
\item $R[X]$ is a (weak) reduced UFR.
\item $R$ is either a UFD or $R=D_1 \times \cdots \times D_n$ where $D_i$ is a UFD with $U(D_i)=\{1\}$.
\end{enumerate}
\end{thm}

\begin{proof}
$(3) \implies (2)$ Here either $R[X]$ is a UFD or $R[X]=D_1[X] \times \cdots \times D_n[X]$ is a UFD with $U(D[X_i])=\{1\}$. In either case it is easily checked that $R[X]$ is a reduced UFR. $(2) \implies (1)$ Clear. $(1) \implies (3)$ Suppose that $R[X]$ is a weak strongly reduced UFR. Then as in $(1) \implies (2)$ of Theorem 5.6 $R[X]$ is factorial and hence $R=D_1 \times \cdots \times D_n$ where $D_i$ is a UFD. Suppose that $n >1$. We need each $U(D_i[X])=U(D_i)=\{1\}$. Now suppose that some $\mid U(D_i) \mid > 1$; let $u \in U(D_i) \backslash \{1\}$. Then $$(0,1,\hdots,1)=(0,1,\hdots,1,u,1,\hdots)(0,1,\hdots,1,u^{-1},1,\hdots1)$$ where $u$ and $u^{-1}$ appear in the $i^{th}$ coordinate are two strongly reduced atomic factorizations of $(0,1,\hdots,1)$, a contradiction. 

\end{proof}

It is interesting to note that while Theorems 3.3 and 3.4 of \cite{CAV} require the factorization of $0$ to be unique, Theorems 5.6 and 5.7 do not. 

Our last result of this section summarizes the various unique factorization characterizations for $R[X]$ and extends them to several variables.

\begin{thm}
Let $R$ be a commutative ring and $X$ an indeterminate over $R$. 
\begin{enumerate} 
\item The following are equivalent.
	\begin{enumerate}
	\item $R$ is a UFD.
	\item $R[X]$ is a UFR (resp., is atomic and any two factorizations of nonzero nonunits into irreducibles are homomorphic).
	\item For any set $\{X_{\alpha}\}_{\alpha \in \Lambda}$, $|\Lambda| \geq 1$, of indeterminates over $R$, $R[\{X_{\alpha}\}_{\alpha \in \Lambda}]$ is a UFR (resp., is atomic and any two factorizations of nonzero nonunits into irreducibles are homomorphic).
	\item For some set $\{X_{\alpha}\}_{\alpha \in \Lambda}$, $|\Lambda| \geq 1$, of indeterminates over $R$, $R[\{X_{\alpha}\}_{\alpha \in \Lambda}]$ is a UFR (resp., is atomic and any two factorizations of nonzero nonunits into irreducibles are homomorphic.)
	\end{enumerate}
	
\item The following are equivalent.
	\begin{enumerate}
	\item $R$ is a finite direct product of UFDs.
	\item $R[X]$ is factorial (resp., a weak UFR, a Fletcher UFR, a (weak) [strong] $\mu$-reduced UFR.)
	\item For any set $\{X_{\alpha}\}_{\alpha \in \Lambda}$, $|\Lambda |\geq 1$, of indeterminates over $R$, $R[\{X_{\alpha}\}_{\alpha \in \Lambda}]$ is factorial (resp., a weak UFR, a Fletcher UFR, a (weak) [strong] $\mu$-reduced UFR).
	\item For some set $\{X_{\alpha}\}_{\alpha \in \Lambda}$, $|\Lambda | \geq 1$, of indeterminates over $R$, $R[\{X_{\alpha}\}_{\alpha \in \Lambda}]$ is factorial (resp., a weak UFR, a Fletcher UFR, a (weak) [strong] $\mu$-reduced UFR).
	\end{enumerate}

\item The following are equivalent.
	\begin{enumerate}
	\item $R$ is either a UFD or $R=D_1 \times \cdots \times D_n$ where each $D_i$ is a UFD with $U(D_i)=\{1\}$. 
	\item $R[X]$ is a (weak) [strongly] reduced UFR.
	\item For any set $\{X_{\alpha}\}_{\alpha \in \Lambda}$, $|\Lambda | \geq 1$, of indeterminates over $R$, $R[\{X_{\alpha}\}_{\alpha \in \Lambda}]$ is a (weak) [strongly] reduced UFR.
	\item For some set $\{X_{\alpha}\}_{\alpha \in \Lambda}$, $|\Lambda | \geq 1$, of indeterminates over $R$, $R[\{X_{\alpha}\}_{\alpha \in \Lambda}]$ is a (weak) [strongly] reduced UFR.
	\end{enumerate}	
\end{enumerate}
\end{thm}

\begin{proof}
\
\begin{enumerate}
\item $(a) \iff (b)$ Corollary 5.5 $(b) \implies (c)$ If $R$ is a UFD, so is $R[\{X_{\alpha}\}_{\alpha \in \Lambda}]$. Hence (c) holds. $(c) \implies (d)$. Clear. $(d) \implies (a)$ Let $\alpha_0 \in \Lambda$ and $\Lambda'=\Lambda \backslash \{\alpha_0\}$. So $R[\{X_{\alpha}\}_{\alpha \in \Lambda'}][X_{\alpha_0}]=R[\{X_{\alpha}\}_{\alpha \in \Lambda}]$ is a UFR (resp., is atomic and any two factorizations of nonzero nonunits into irreducibles are homomorphic). By $(a) \iff (b)$, $R[\{X_{\alpha}\}_{\alpha \in \Lambda'}]$ is a UFD. Hence $R$ is a UFD.

\item $(a) \iff (b)$ Theorems 5.4 and 5.6. The other implications follow as in $(1)$ mutatis mutandis. 

\item $(a) \iff (b)$ Theorem 5.7. The other implications follow as in (2) mutatis mutandis and the observation that for any set $\{X_{\alpha}\}_{\alpha \in \Lambda}$ of indeterminates over an integral domain $D$, $U(D)=U(D[\{X_{\alpha}\}_{\alpha \in \Lambda}])$.

\end{enumerate}
\end{proof}


\section{Polynomial Rings over Unique Factorization Rings}

A polynomial ring over a UFD is again a UFD. This is not the case for the other types of unique factorization rings that we have defined. Suppose that $R$ is a total quotient ring. Then $R$ is certainly factorial, but $R[X]$ need not even be atomic; indeed, $X$ need not be a finite product of atoms (e.g., $R$ an infinite direct product of fields). If $R$ is an SPIR or a quasilocal ring $(R,M)$ with $M^2=0$, then $R$ is factorial, a UFR, a weak UFR, and a [strongly] ($\mu$-) reduced UFR, but $R[X]$ has none of those properties except for the trivial case where $R$ is a field. And if $R$ is an SPIR, then $R$ is a Fletcher UFR, but $R[X]$ is not a Fletcher UFR unless $R$ is a field. This raises the question of what factorization properties $R[X]$ has in the case where $R$ is an SPIR or a quasilocal ring $(R,M)$ with $M^2=0$. 

Recall the following definitions from \cite{AV1}. Let $R$ be a commutative ring. Then $R$ is a \emph{half factorial ring (HFR)} if $R$ is atomic and for any nonzero nonunit $a \in R$, any two atomic factorizations of $a$ have the same length. The ring $R$ is a \emph{bounded factorization ring (BFR)} if (1) $R$ is atomic and (2) for each nonzero nonunit $a \in R$, there is a bound on the length of atomic factorizations of $a$, or equivalently, (3) for each nonzero nonunit $a \in R$, there is a natural number $N(a)$ so that for any factorization of $a$ into nonunits $a=a_1 \cdots a_n$, $n \leq N(a)$. Finally, $R$ is called a \emph{finite factorization ring (FFR)} if every nonzero nonunit of $R$ has only a finite number of factorizations up to order and associates, a \emph{weak finite factorization ring (WFFR)} if every nonzero nonunit of $R$ has only a finite number of nonassociate divisors, and an \emph{atomic idf ring} if $R$ is atomic and each nonzero element of $R$ has at most a finite number of nonassociate irreducible divisors. An HFR, FFR, and a BFR are pr\'esimplifiable. 

Note that the following are equivalent: (1) $R$ is an FFR, (2) $R$ is a BFR and WFFR, (3) $R$ is pr\`esimplifable and a WFFR, (4) $R$ is a BFR and an atomic idf-ring, and (5) $R$ is pr\`esimplifiable and an atomic idf ring \citep[Proposition 6.6]{AV1}. We have the following diagram where none of the implications can be reversed.

\

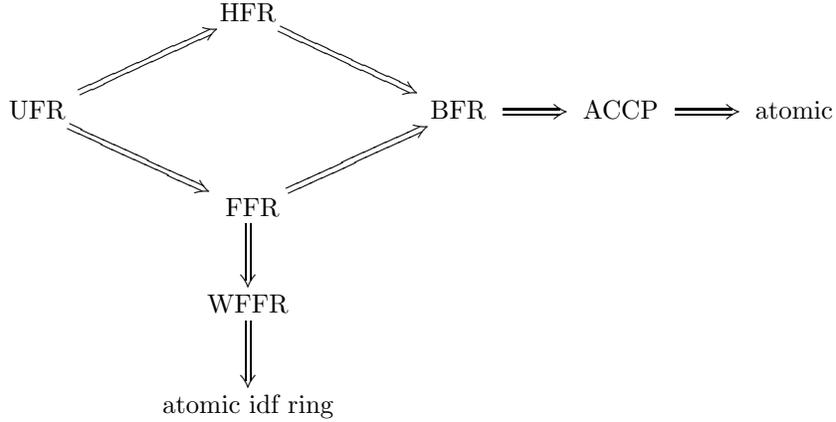
\vspace{-5mm}
\[
\xymatrix{ 
& \text{ HFR } \ar@{=>}[dr]\\
\text{ UFR } \ar@{=>}[ur] \ar@{=>}[dr] & &\text{ BFR } \ar@{=>}[r] &\text{ ACCP } \ar@{=>}[r] &\text{ atomic }\\
&\text{ FFR} \ar@{=>}[ur] \ar@{=>}[d]\\
& \text{ WFFR } \ar@{=>}[d]\\
& \text{ atomic idf ring }}\\
\] 

\

In Section 3 we remarked that Coykendall and Trentham \cite{CT} gave an example of a zero-dimensional quasilocal ring $S$ with $S[X]$ atomic (or equivalently, very strongly atomic), but $S$ is not atomic. It is easily checked that if $R[X]$ satisfies any of the conditions in the diagram other than being atomic or an atomic idf ring, then so does $R$. Certainly, if $R[X]$ is an idf ring, so is $R$. Now if $R[X]$ is an atomic idf ring, then either $R$ is an integral domain (and hence a FFD) or $R$ is a finite local ring \citep[Theorem 1.7]{AV2}. In either case $R$ is pr\`esimplifiable and hence so is $R[X]$. So for a polynomial ring $R[X]$, the notions of FFR, WFFR, and atomic idf ring coincide. Hence if $R[X]$ is an atomic idf ring, so is $R$.


Suppose that $(R,M)$ is a quasilocal ring with $M^n=0$ for some $n \geq 1$. Then $R$ is a BFR. We next note that $R[X]$ is a BFR.

\begin{thm}
Suppose that $(R,M)$ is a quasilocal ring with $M^n=0$ for some $n \geq 1$. Let $X$ be an indeterminate over $R$. Then $R[X]$ is a BFR. Hence if $R$ is a UFR, $R[X]$ is a BFR and thus is atomic.
\end{thm}

\begin{proof}
In \citep[Theorem 12]{BFR} it was shown that if $R$ is a BFR with the zero ideal primary and $\bigcap_{m=1}^{\infty} (nil(R))^m=0$, then $R[X]$ is a BFR. Thus if $(R,M)$ is a quasilocal ring with $M^n=0$ for some $n \geq 1$, $R[X]$ is a BFR.

\end{proof}

We next show for $R$ a UFR, $R[X]$ is an HFR only in the trivial case where $R$ is a UFD.

\begin{thm} 
Let $R$ be a commutative ring and $X$ an indeterminate over $R$. Suppose that $R$ has a nonzero nilpotent atom. Then $R[X]$ is a not an HFR. Hence if $R$ is a zero-dimensional quasilocal ring that is not a field (e.g., a UFR that is not a UFD), $R[X]$ is not an HFR.
\end{thm}


\begin{proof}
We may suppose that $R$ is indecomposable for otherwise $R[X]$ is not an HFR. Let $m$ be a natural number. Let $b \in nil(R)$. We first note that $X^m+b$ is a product of atoms and any atomic factorization of $X^m+b$ has at most $m$ factors. For let $X^m+b=f_1 \cdots f_n$ where $f_i \in R[X]$ is a nonunit. Pass to $\bar{R}=R/nil(R)$. So $\bar{R}$ is an indecomposable reduced ring. Then $X^m=\overline{X^m+b}=\bar{f_1}\cdots \bar{f_n}$. Since $f_i$ is a nonunit, so is $\bar{f_i}$. By Corollary 4.6 $n \leq m$. Thus $X^m+b$ is a product of atoms and any atomic factorization of $X^m+b$ has at most $m$ factors. 

Suppose further that $b$ is an atom, then $X^m+b$ is actually irreducible. For let $X^m+b=f_1 \cdots f_n$ where each $f_i \in R[X]$ is irreducible. Since $b$ is an atom exactly one $f_i$ has $f_i(0)$ a nonunit. Suppose that $f_j(0)$ is a unit. Then $\overline{f_j(0)}$ is a unit and $\bar{f_j}$ is a factor of $X^m$, so again by Corollary 4.6, $\bar{f_j}$ is a unit and so $f_j$ is a unit. So $X^m+b$ is an atom. Choose $n \geq 1$ with $b^{2^{n-1}}=0$. Then $X^{2^n}=X^{2^n}-b^{2^{n-1}}=(X^{2^{n-1}}+b^{2^{n-2}})(X^{2^{n-2}}+b^{2^{n-3}})\cdots(X^4+b^2)(X^2+b)(X^2-b)$. Since $R$ is indecomposable, $X$ is an atom. Hence $X^{2^n}$ is a product of $2^n$ atoms. Factoring $X^{2^{n-1}}+b^{2^{n-2}}, \cdots, X^4+b^2$ into atoms and noting that $X^2+b$ and $X^2-b$ are both atoms we see that $X^{2^n}=X^{2^n}-b^{2^{n-1}}$ has an atomic factorization with at most $2^{n-1}+2^{n-2}+\cdots+4+2 < 2^n$ atoms. Thus $R[X]$ is not a HFR.

\end{proof}

For other examples of non-unique factorization in the SPIR $\mathbb{Z}_{p^n}$ for $n \geq 2$, see \cite{FF}.\\

We end by noting that for $R$ a UFR that is not an integral domain (or equivalently a UFD), $R[X]$ is an FFR if and only if $R$ is finite.

\begin{thm}
Let $R$ be a commutative ring and $X$ an indeterminate over $R$. Then $R[X]$ is an FFR if and only if either (1) $R$ is an FFD or (2) $(R,M)$ is a finite local ring satisfying (a) for $x_1,\hdots,x_k \in M \backslash M^2$, $x_1\cdots x_k \in M^k \backslash M^{k+1}$ for $1 \leq k \leq n-1$ (or equivalently, $x \in M^i \backslash M^{i+1}$ and $y \in M^j \backslash M^{j+1}$ $\implies xy \in M^{i+j} \backslash M^{i+j+1}$  for $i+j < n$) where $M^n=0$, but $M^{n-1}\neq0$ and (b) $aM=M^2$ for $a \in M \backslash M^2$. Thus if $R$ is a UFR, $R[X]$ is an FFR if and only if $R$ is a UFD or $R$ is finite. 
\end{thm}

\begin{proof}
First suppose that $(R,M)$ is a finite local ring where $M^n=0$, but $M^{n-1} \neq 0$. Then \citep[Theorem 17]{Alan} $R[X]$ is an FFR if and only if $R$ satisfies $(a')$ if $r_1,\hdots, r_k$ are atoms of $R$ where $k<n$, then $r_1\cdots r_k \in M^k \backslash M^{k+1}$ and $(b')$ every element of $M \backslash M^2$ divides all elements of $M^2$. Clearly $(b)$ and $(b')$ are equivalent. Also $(b)$ implies that $r \in R$ is an atom if and only if $r \in M \backslash M^2$. So $(a)$ and $(a')$ are equivalent. Also, note that $x \in M^k \backslash M^{k+1}$ for $1 \leq k \leq n-1$ if and only if $x$ is a product of $k$ atoms. So $(a)$ is also equivalent to the condition: $x \in M^i \backslash M^{i+1}$ and $y \in M^j \backslash M^{j+1}$ where $i+j < n$ $\implies$ $xy \in M^{i+j} \backslash M^{i+j+1}$. 

\noindent $(\Rightarrow)$ Suppose that $R[X]$ is an FFR. By \citep[Theorem 1.7]{AV2}, $R$ is either an integral domain (and hence a FFD) or $R$ is a finite local ring. But if $R$ is a finite local ring, then $R$ satisfies (a) and (b) by the remarks of the previous paragraph.

\noindent $(\Leftarrow)$ If $R$ is an FFD, then it is well known that $R[X]$ is an FFD \citep[Proposition 5.3]{AAZ}. If $R$ is a finite local ring satisfying (a) and (b), $R[X]$ is an FFR by the remarks of the first paragraph of the proof.

The last statement is now immediate since a quasilocal ring $(R,M)$ with $M^2=0$ or $R$ an SPIR clearly satisfies (a) and (b). (This is remarked in the paragraph after \citep[Theorem 17]{Alan}.)

\end{proof}

Note that if $R$ is one of the types of ``unique factorization rings" that are not indecomposable, then $R$ and $R[X]$ have nontrivial idempotents and hence are not BFRs, let alone HFRs and FFRs.

\

\

\end{document}